\documentclass{article}
\usepackage[utf8]{inputenc}
\usepackage{amsmath,amsthm,fullpage,graphicx}
\usepackage{hyperref}
\usepackage{listings}
\usepackage{authblk}
\usepackage{xurl}
\newtheorem{theorem}{Theorem}
\newtheorem{lemma}{Lemma}

\newtheorem{proposition}{Proposition}

\newcommand{\oeis}[1]{\href{https://oeis.org/#1}{#1}}

\title{Faulhaber's formula, Bernoulli numbers, power sums of natural numbers and totatives and the functional equation $f(x)+x^k=f(x+1)$}
\author{Chai Wah Wu}
\affil{IBM Research\\IBM T. J. Watson Research Center, Yorktown Heights, NY, USA\thanks{cwwu@us.ibm.com}}

\date{March 1, 2024\\Latest update: November 29, 2026}

\begin{document}

\maketitle

\begin{abstract} In modern usage the Bernoulli numbers and Bernoulli polynomials follow Euler's approach and are defined using generating functions. Originally, they were derived by Bernoulli while characterizing Faulhaber's formula for the sum of consecutive powers. 
These equations
have many consequences and applications in various areas of mathematics. We consider yet another application by studying the functional equation $f(x)+x^k=f(x+1)$ and show that a solution of this equation can be derived from Faulhaber's formula. 
We then use these results to study the totatives of $n$, i.e. numbers less than $n$ that are coprime to $n$. In particular, we look at sums of powers of totatives of $n$ that are less than $\frac{n}{2}$.
We show that
the sum of powers of this half of the totatives can also be expressed in the same structural form as the sum of powers of all totatives
and provide explicit formulas for this sum. As an application of these results, we obtain a formula for the total area of all rectangles with coprime width and length and semiperimeter $n$. 
\end{abstract}

\section{Faulhaber's formula and Bernoulli numbers}
Faulhaber's formula (also known as Bernoulli's formula\footnote{Some authors prefer the term Bernoulli's formula as Jacob Bernoulli was one of the first to give a mathematical study of this equation's mathematical properties. However, this use might lead to confusion with a fundamental equation in fluid dynamics named after his nephew Daniel Bernoulli.}) \cite{Apostol1976} equates the sum of $k$-th powers of $n$ consecutive integers to a $(k+1)$-th degree polynomial of $n$. It was discovered by various people including Seki, Faulhaber, and Bernoulli.
The coefficients of this polynomial (after suitable scaling) are now known as Bernoulli numbers.
In particular, for $n$ and $k$ nonnegative integers, Faulhaber's formula is written as
\begin{equation}
\label{eqn:faulhaber} 
F(n,k) = \sum_{i=1}^n i^k = \frac{1}{k+1}\sum_{j=0}^k\binom{k+1}{j}B^+_j n^{k+1-j}
\end{equation}
where $B^+_j$ are known as the Bernoulli numbers of the second kind. 
There are in general two conventions (resulting in two kinds of numbers) when defining Bernoulli numbers, with the only difference being the sign of $B_1$. In particular, $B^+_1=\frac{1}{2}$, whereas
the Bernoulli numbers of the first kind $B^-_j$ are identical to $B^+_j$ except for $B^-_1 = -\frac{1}{2}$. The convention of using $B^-_j$ to denote Bernoulli numbers has been used in modern textbooks although recently there are discussions\footnote{\url{https://luschny.de/math/zeta/The-Bernoulli-Manifesto.html}, \url{https://www-cs-faculty.stanford.edu/~knuth/news22.html}} on whether $B^+_j$, which was used in the past, should be favored. Computer algebra systems (CAS) such as \verb+Mathematica+, \verb+Maple+, and \verb+PARI+ currently use $B^-_j$ to denote Bernoulli numbers (as of this writing). 
Until very recently, the \verb+bernoulli+  function in the Python-based CAS \verb+Sympy+ also returns $B^-_j$. 
Since version 1.12, however, this function returns $B^+_j$. Note that the number of terms in the left hand side of Eq. (\ref{eqn:faulhaber}) is $n$ while it is $k+1$ on the right hand side, meaning that the computation of $F(n,k)$ is more efficient using the right hand side when $n\gg k$.

Whereas the Bernoulli numbers originally arises from Bernoulli's study of Faulhaber's formula, they are defined in modern usage via the generating functions  \cite{Knuth1993}:
\begin{eqnarray}
\sum_{j=0}^\infty \frac{B^-_jx^j}{j!} = \frac{x}{e^x-1} \label{eqn:bernoulli1}\\
\sum_{j=0}^\infty \frac{B^+_jx^j}{j!} = \frac{x}{1-e^{-x}} \label{eqn:bernoulli2}
\end{eqnarray}

They also satisfy the following recursive definitions \cite{Knuth1993}:
\begin{eqnarray}
\sum_{j=0}^m\binom{m+1}{j}B^-_j = \delta_{m,0} \label{eqn:bernoulli_sum_-}\\
\sum_{j=0}^m\binom{m+1}{j}B^+_j = m+1 \label{eqn:bernoulli_sum_+}
\end{eqnarray}
where $\delta_{m,0}$ is the Kronecker delta function defined as $\delta_{m,0}=1$ if $m=0$ and $\delta_{m,0}=0$ otherwise.
A related important concept is the set of Bernoulli polynomials $\hat{B}_k$ defined as:
$$\hat{B}_k(x) = \sum_{j=0}^{k} \binom{k}{j}B_j^- x^{k-j}$$
Comprehensive introductions to the Bernoulli numbers and Bernoulli polynomials including their many properties can be found in \cite{Abramowitz1964,Apostol1976,Hardy2008}.

\section{Variations of Faulhaber's formula}
Faulhaber's formula as described in Eq. (\ref{eqn:faulhaber}) expresses the sum of powers $F(n,k)$ using Bernoulli numbers of the second kind $B^+_j$. The following formula gives a variant of Faulhaber's formula using Bernoulli numbers of the first kind $B^-_j$ \cite{larson:bernoulli:2019}.
\begin{theorem}
\label{thm:faulhaber-new}
If $k>0$, then
\begin{equation}
\label{eqn:faulhaber-new} F(n,k) = \sum_{i=1}^n i^k = \frac{1}{k+1}\sum_{j=0}^k\binom{k+1}{j}B^-_j (n+1)^{k+1-j}\end{equation}
\end{theorem}
\begin{proof}
By Eq. (\ref{eqn:faulhaber}),
$$ F(n+1,k) = F(n,k)+(n+1)^k = \frac{1}{k+1}\sum_{j=0}^k\binom{k+1}{j}B^+_j (n+1)^{k+1-j}$$
The two kinds of Bernoulli numbers differ only at $j=1$, and we note that the term in the summation above corresponding to $j=1$ is $\frac{1}{2}(n+1)^k$ whereas $\frac{1}{k+1}\binom{k+1}{1}B^-_1(n+1)^k = -\frac{1}{2}(n+1)^k$.
This means that 
$$F(n+1,k) =   (n+1)^k+ \frac{1}{k+1}\sum_{j=0}^k\binom{k+1}{j}B^-_j (n+1)^{k+1-j}$$ and the conclusion follows.
\end{proof}
Note that the formula in Eq. (\ref{eqn:faulhaber-new}) is almost identical to the formula in Eq. (\ref{eqn:faulhaber}) except for using $B^-_j$ rather than $B^+_j$ and shifting the base of the powers from $n$ to $n+1$. A trivial extension is the following generalization which we will use later:

\begin{theorem}
\label{thm:faulhaber-new-2}
Let $h$ be a function on the integers.
If $k>0$ and $h(n)$ is an integer, then
\begin{eqnarray}
\label{eqn:faulhaber-new-2} F(h(n),k) = \sum_{i=1}^{h(n)} i^k &=& \frac{1}{k+1}\sum_{j=0}^k\binom{k+1}{j}B^+_j h(n)^{k+1-j}\\
&=& \frac{1}{k+1}\sum_{j=0}^k\binom{k+1}{j}B^-_j (h(n)+1)^{k+1-j}
\end{eqnarray}
\end{theorem}

The Bernoulli numbers are rational numbers and the first few values are given by $1$, $\pm\frac{1}{2}$, $\frac{1}{6}$, $0$,
$-\frac{1}{30}$,$\dots$
There is a bit of irregularity in having two kinds of Bernoulli numbers which differs only at odd index $j=1$, especially since they are zero at all other odd indices $j$. It is possible to ignore $B_1$ in many of the equations concerning $B_j$, or equivalently by considering $\tilde{B}_j = \frac{B_j^+ + B_j^-}{2}$ instead. For instance, 
a consequence of Theorem \ref{thm:faulhaber-new} is the following variation of Faulhaber's formula that is satisfied by both kinds of Bernoulli numbers:
$$ F(n-1,k)+F(n,k) = \frac{2}{k+1}\sum_{j=0, j\neq 1}^k\binom{k+1}{j}B^\pm_j n^{k+1-j}$$
where $B^{\pm}_j$ means that Bernoulli numbers of either kind can be used. This is true in the formula above as $B^{\pm}_1$ is not used.
Similarly, 
Eqns. (\ref{eqn:bernoulli1})-(\ref{eqn:bernoulli2}) can be combined into
$$\sum_{j=0,j\neq 1}^\infty \frac{B^\pm_jx^j}{j!} = \frac{x}{2}\coth\left(\frac{x}{2}\right)$$
or equivalently
$$\sum_{j=0,j\neq 1}^\infty \frac{2^jB^\pm_jx^j}{j!} = x\coth\left(x\right)$$
and for $m>0$ the $2$ equations 
Eqns. (\ref{eqn:bernoulli_sum_-})-(\ref{eqn:bernoulli_sum_+}) 
can be combined into
$$ 
\sum_{j=0,j\neq 1}^m\binom{m+1}{j}B^\pm_j = \frac{m+1}{2}.$$

We will show in a later section another formula involving Bernoulli numbers where either kind can be used. 

\section{A solution to the equation $f(x)+x^k=f(x+1)$}
Consider the functional equation 
\begin{equation}\label{eqn:feq} f(x)+x^k=f(x+1)
\end{equation}
and define $f_k(x) = \frac{1}{k+1}\sum_{j=0}^k\binom{k+1}{j}B^-_j x^{k+1-j}$.  The above discussion shows that $f_k$ satisfies Eq. (\ref{eqn:feq}), i.e., $f_k(x+1) = f_k(x)+x^k$. Note that $(k+1)f_k(x) = \hat{B}_{k+1}(x)-B_{k+1}^-$. 
This implies the following well-known property of Bernoulli polynomials: $\hat{B}_k(x+1)-\hat{B}_k(x) = kx^{k-1}$ (see e.g. \cite{larson:bernoulli:2019,Ireland1982}), and leads to the following form of Faulhaber's formula: $F(n,k) = \frac{1}{k+1}(\hat{B}_{k+1}(n+1)-\hat{B}_{k+1}(0))$.
We show that the converse is also true in the following sense.

Suppose that $p$ is a minimal degree polynomial that satisfied Eq. (\ref{eqn:feq}). Clearly $p$ must have a degree larger than or equal to $k$.
If $p$ has degree $k$, then the coefficient of $x^k$ does not match in Eq. (\ref{eqn:feq}). Therefore, $p$ must be of degree $k+1$ or larger.
Consider the case where $p$ has degree $k+1$ written as $p(x) = \sum_{j=0}^{k+1}a_ix^j$.
Then Eq. (\ref{eqn:feq}) can be written as
$\sum_{j=0}^{k+1}a_j x^j+x^k = \sum_{j=0}^{k+1}a_j (x+1)^j$.
Matching the coefficients results in the following $k+2$ equations
$a_j = \sum_{l=0}^{k+1}a_l\binom{l}{j}= \sum_{l=j}^{k+1}a_l\binom{l}{j}$ for $j\neq k$
and $a_k+1 =   \sum_{l=0}^{k+1}a_l\binom{l}{k}$. This last equation can be simplified as:
$a_k+1=a_k+a_{k+1}(k+1)$, i.e. $a_{k+1} = \frac{1}{k+1}$.
The equation for $j=k+1$ is the identity $a_{k+1} =a_{k+1}$.
The remaining $k$ equations $a_j = \sum_{l=j}^{k+1}a_l\binom{l}{j}$ for $j<k$ uniquely determine the other coefficients $a_j$. In particular, $a_j = a_j+a_{j+1}\binom{j+1}{j} + \sum_{l=j+2}^{k+1}a_{l}\binom{l}{j}$,
i.e. for $j\leq k$,
$$a_{j} = -\frac{\sum_{l=j+1}^{k+1}a_{l}\binom{l}{j-1}}{\binom{j}{j-1}}= -\frac{\sum_{l=j+1}^{k+1}a_{l}\binom{l}{j-1}}{j}$$
This implies that
$a_k=-a_{k+1}\binom{k+1}{k-1}/k= -\frac{1}{2}$. 

Let $b_{k+1-j} = (k+1)a_j/\binom{k+1}{j}$ or equivalently
$b_{j} = (k+1)a_{k+1-j}/\binom{k+1}{j}$. This implies that $b_0 = 1$ and $b_1 = -\frac{1}{2}$.
After straightforward algebraic manipulations we get the following identities regarding 
$b_j$'s:
$b_j = -\frac{\sum_{w=0}^{j-1} b_w \binom{j+1}{w}}{j+1}$ for $0<j\leq k+1$ which is exactly the recursive definition of the Benoulli number $B^-_j$ in 
Eq. (\ref{eqn:bernoulli_sum_-}).
Thus we have shown that 
\begin{proposition} \label{prop:minimal1}
The first $k+2$ Bernoulli numbers $B^-_j$ are defined by the coefficients $a_j$ of the minimal degree polynomial that satisfies the functional equation $f(x)+x^k=f(x+1)$ via the equation $B^-_{j} = (k+1)a_{k+1-j}/\binom{k+1}{j}$.
\end{proposition}
Define $g(x) = \frac{1}{k+1}\sum_{j=0}^k\binom{k+1}{j}B^+_j x^{k+1-j}$, a $(k+1)^{th}$-degree polynomial with rational coefficients. Then Eqns (\ref{eqn:faulhaber}),(\ref{eqn:faulhaber-new}) show that
$g(x)+(x+1)^k =g(x+1)$. 
An analogous derivation shows that
\begin{proposition} \label{prop:minimal2}
The first $k+2$ Bernoulli numbers $B^+_j$ are defined by the coefficients $a_j$ of the minimal degree polynomial that satisfies the functional equation $f(x)+(x+1)^k=f(x+1)$ via the equation  $B^+_{j} = (k+1)a_{k+1-j}/\binom{k+1}{j}$.
\end{proposition}

Since the Bernoulli polynomials $\hat{B}_k(x)$ form an Appell sequence \cite{Roman1984}, i.e., they satisfy the differential equation $\frac{d}{dx}\hat{B}_k(x) = k\hat{B}_{k-1}(x)$, and it is clear that
$f_k$ is an Appell sequence as well. This implies that it satisfies Appell's identity: $f_k(x+y) = \sum_{i=0}^k\binom{k}{i}f_i(x)y^{k-i}$, i.e.  $f_k(x)+x^k = f_k(x+1) = \sum_{i=0}^k \binom{k}{i}f_{i}(x)$. This simplifies to the equation
$\sum_{i=0}^{k-1}\binom{k}{i}f_i(x) = x^k$.
This equation can also be derived from the well-known identity $kx^{k-1} =\sum_{i=1}^k\binom{k}{i}\hat{B}_{k-i}(x)$.

\section{Sums of powers of totatives}\label{sec:totatives}
Faulhaber's formula allows the sum of $k$-th powers of the first $n$ positive integers to be written as a polynomial of degree $k+1$. In this section we look at an application of Faulhaber's formula to find the sum of $k$-th powers of the totatives of $n$ \cite{Sylvester1879,Dickson1919,Polya1925,Brown2005}.

We will use the following standard notations from number theory. Let $J_k(n) = n^k\prod_{p|n, p\text{ prime}}\left(1-\frac{1}{p^k}\right)$ be the $k$-th degree Jordan totient of $n$, $\text{rad}(n)=\prod_{p|n,p\text{ prime}}p$ be the radical or squarefree kernel of $n$, $\nu(n)$ be the number of distinct prime factors of $n$, and $J^{-1}_k$ be the Dirichlet inverse of $J_k$ defined recursively as $J^{-1}_k(n) = -\sum_{d|n,d<n}J_k(\frac{n}{d})J^{-1}_k(d)$. Depending on context, $\gcd(a,b)$ is simplified with the shorthand $(a,b)$. Let $\phi(n)$ denote the number of totatives of $n$ (also known as Euler's phi function or as the totient function) and $\mu(n)$ denote the M\"{o}bius function, i.e. $\mu(1)=1$, $\mu(n) = 0$ if $n$ is not squarefree and $\mu(n)=(-1)^{\nu(n)}$ otherwise.
Let $\psi_k(n) = \prod_{p|n, p \text{ prime}} 1-p^k$ with $\psi_k(1)=1$. 
It is easy to show that $\psi_0(n) = 0$ for $n>1$, $\phi(n) = n\psi_{-1}(n)$ and $\psi_k(n) = J^{-1}_k(n) = \frac{(-1)^{\nu(n)}J_k(n)(\text{rad}(n))^k}{n^k}$.
Since $d^k$ is multiplicative in $d$, it is straightfoward \cite{Brown2005} to show that 
\begin{equation}\label{eqn:psi}\psi_k(n) = \sum_{d|n}\mu(d)d^k
\end{equation} 
Let $S_k(n) = \sum_{(d,n)=1} d^k$ denote the sum of the $k$-th powers of the totatives of $n$.
In \cite{Brown2005}, the following formula for $S_k$ was shown by means of the M\"{o}bius inversion formula \cite{Apostol1976}:
$$ S_k(n) = n^k\sum_{d|n}\frac{1}{d^k}\mu\left(\frac{n}{d}\right)\sum_{i=1}^d i^k$$
When $k$ is smaller than $d$, it is useful to replace the sum $F(d,k) = \sum_{i=1}^d i^k$ with a polynomial of $d$ with $k+1$ terms by means of Faulhaber's formula.
Furthermore, by using Eq. (\ref{eqn:psi}) and swapping the summation, the function $S_k(n)$ for $n>1$ can be written as:
$$ S_k(n) = \frac{n^k\phi(n) + \sum_{j=2}^{k} n^{k+1-j}\binom{k+1}{j}B_j^+\psi_{j-1}(n)}{k+1}$$
Since $B_{1}^+$ is not used in the above formula, either kind of Bernoulli numbers ($B_i^+$ or $B_i^-$) could have been used.
This can be more compactly written as:
\begin{equation}\label{eqn:sum1a}
S_k(n) = \sum_{-1\leq i< k, i\text{ odd}} a_{i}^kn^{k-i}\psi_{i}(n)
\end{equation}
with
$a_{i}^k = \frac{\binom{k+1}{i+1}B^\pm_{i+1}}{k+1}$.

Note that $a_i^k$ are exactly the coefficients of the minimal polynomials in either Proposition \ref{prop:minimal1} or Proposition \ref{prop:minimal2} and that $(k+1)a_i^k$ are the coefficients of the Bernoulli polynomial $\hat{B}_{k+1}$.

For the first few values of $k$, $S_k$ has the following forms \cite{Brown2005} for $n>1$:
$S_1(n) = \frac{n\phi(n)}{2}$,
$S_2(n) = \frac{n^2}{3}\phi(n)+\frac{n}{6}\psi_1(n)$, 
$S_3(n) = \frac{n^3}{4}\phi(n)+\frac{n^2}{4}\psi_1(n)$, and
$S_4(n) = \frac{n^4}{5}\phi(n)+\frac{n^3}{3}\psi_1(n)-\frac{n}{30}\psi_3(n)$ and these can be found in the Online Encyclopedia of Integer Sequences (OEIS) \cite{oeis} as sequences \oeis{A023896}, \oeis{A053818}, \oeis{A053819}, \oeis{A053820}, respectively.

\section{Sum of powers of the first half of totatives}\label{sec:totativesR}
Since $\gcd(x,n)=1$ if and only if $\gcd(n-x,n)=1$, the set of totatives is naturally partitioned into pairs $\{x,n-x\}$ and thus it makes sense to consider $R_n=\{1\leq d\leq \frac{n}{2}|(d,n)=1\}$, the set of totatives that are less than or equal to $\frac{n}{2}$\footnote{Note that for even $n>2$, $n/2$ is not a totative since $\gcd(n,n/2) = n/2$ and thus $R_n=\{1\leq d< \frac{n}{2}|(d,n)=1\}$ and $|R_n| = \phi(n)/2$ for $n>2$.}. Since for $n>2$, there are $\phi(n)/2$ such pairs and each pair sums to $n$, this can be used to show that $S_1(n) = n\phi(n)/2$.
Let us define $\tilde{S}_k(n) = \sum_{R_n} d^k$.

Baum \cite{Baum1982} provided formulas for $\tilde{S}_1(n)$ and $\tilde{S}_2(n)$. His method can be extended to $\tilde{S}_k(n)$ for general values of $k$ and in this section we provide formulas for $\tilde{S}_k$ for all $k$ and illustrate structural features of $\tilde{S}_k$. In particular we show that just like $S_k$,  $\tilde{S}_k$ can also be written as linear combinations of Dirichlet inverses of Jordan totients of odd degrees.

Let $\alpha_k(n) = \frac{1}{n^k}\sum_{i=1}^{\lfloor\frac{n-1}{2}\rfloor} i^k$.
As we will only use $\alpha_k$ when $n$ is odd, let us assume that $n$ is odd in all the discussion about $\alpha_k$. In this case 
$\lfloor\frac{n-1}{2}\rfloor = \frac{n-1}{2}$.
By Faulhaber's formula $\alpha_k$ can be written as:
\begin{eqnarray}
 \alpha_k(n) & = &\frac{1}{(k+1)n^k}\sum_{j=0}^k\binom{k+1}{j}B_j^+ \frac{(n-1)^{k+1-j}}{2^{k+1-j}}\\
 &=& \frac{1}{k+1}\sum_{j=0}^k\binom{k+1}{j}\frac{B_j^+}{2^{k+1-j}} \sum_{i=0}^{k+1-j}(-1)^i\binom{k+1-j}{i}n^{1-j-i} \label{eqn:alpha_k_expand}
 \end{eqnarray}
 
By Theorem \ref{thm:faulhaber-new-2} this can also be written as:
$$\alpha_k(n) = \frac{1}{(k+1)n^k}\sum_{j=0}^k\binom{k+1}{j}B_j^- \frac{(n+1)^{k+1-j}}{2^{k+1-j}}$$
 
The following result shows that
$\alpha_k(n)$ can be written as a polynomial with powers $n^j$ where $j$ is odd and ranges from $-k$ to $1$.  This implies that $\alpha_k(n/d)$ can be written as a polynomial of $d$
with powers $d^j$ where $j$ is odd and ranges from $-1$ to $k$. 

\begin{lemma} \label{lem:alpha_odd}
If $k$ is even, then $\alpha_k(n) = n^{1-k}(n^2-1)P_k(n^2)$ for some polynomial $P_k$ with rational coefficients.
If $k$ is odd, then $\alpha_k(n) = n^{-k}P_k(n^2)$ for some polynomial $P_k$ with rational coefficients. In other words,
all nonzero coefficients of $\alpha_k(n)$ as a polynomial of $n$ occur at odd powers.
\end{lemma}

\begin{proof}
Using Faulhaber's formula in terms of Bernoulli polynomials $\hat{B}_i(n)$,
$\alpha_k(n)$ can be written as 
$$\alpha_k(n) = \frac{1}{n^k(k+1)}\left(\hat{B}_{k+1}\left(\frac{n+1}{2}\right)-\hat{B}_{k+1}(0)\right)$$.

We make use of the following symmetry property of Bernoulli polynomials: $\hat{B}_k(1-x) = (-1)^k\hat{B}_k(x)$ \cite{Abramowitz1964}, i.e.
if $k$ is even, then $\hat{B}_{k+1}(x)$ is an odd function around $x=\frac{1}{2}$ and an even function around $x=\frac{1}{2}$ otherwise.
This means that if $k$ is even, then $\hat{B}_{k+1}(x+\frac{1}{2})$ is an odd function of $x$, i.e $\hat{B}_{k+1}(x+\frac{1}{2}) = xQ_{k+1}(x^2)$ for some polynomial $Q_{k+1}$.
If $k>0$ is even, then $\hat{B}_{k+1}(0) = \hat{B}_{k+1}(1) = 0$, i.e. $x=\pm \frac{1}{2}$ are roots and thus $(x-\frac{1}{2})(x+\frac{1}{2})$ is a divisor of $\hat{B}_{k+1}(x+\frac{1}{2})$. 
Similarly if $k$ is odd, then $\hat{B}_{k+1}(x+\frac{1}{2}) = Q_{k+1}(x^2)$ for some polynomial $Q_{k+1}$. The conclusions follow after substituting $x=\frac{n}{2}$.
\end{proof}

By grouping the terms in Eq. (\ref{eqn:alpha_k_expand}), we can write $\alpha_k(n)$ as $$\alpha_k(n) = \sum_{-k\leq i\leq 1, i \text{ odd}} b_i^k n^i$$
where the rational coefficients $b_{i}^k$ are given by:

\begin{eqnarray}
b_i^k &=& \frac{1}{(k+1)2^{k+1}}\sum_{j=0}^{\min(k,1-i)}(-1)^{1-j-i} 2^j B_{j}^{+}\binom{k+1}{j}\binom{k+1-j}{1-j-i} \label{eqn:bik}
\end{eqnarray}
or in terms of trinomial coefficients:
\begin{eqnarray}
b_i^k &=& \frac{1}{(k+1)2^{k+1}}\sum_{j=0}^{\min(k,1-i)}(-1)^{1-j-i} 2^j B_{j}^{+}\binom{k+1}{j,\, (1-j-i),\, (k+i)}
\end{eqnarray}

For instance,
$\alpha_k(n)$ for the first few values of $k$ are given by:

\begin{itemize}
\item $\alpha_1(n) =  \frac{n - n^{-1}}{8} $.
\item $\alpha_2(n) =  \frac{n - n^{-1}}{24} $.
\item $\alpha_3(n) =  \frac{n - 2 n^{-1} + n^{-3}}{64} $.
\item $\alpha_4(n) =  \frac{3 n - 10 n^{-1} + 7 n^{-3}}{480} $.
\item $\alpha_5(n) =  \frac{n - 5 n^{-1} + 7 n^{-3} - 3 n^{-5}}{384} $.
\item $\alpha_6(n) =  \frac{3 n - 21 n^{-1}+ 49 n^{-3} - 31 n^{-5}}{2688 } $.
\item $\alpha_7(n) =  \frac{3 n - 28 n^{-1} + 98 n^{-3} - 124 n^{-5} + 51 n^{-7}}{6144 } $.
\item $\alpha_8(n) =  \frac{5 n - 60 n^{-1} + 294 n^{-3} - 620 n^{-5} + 381 n^{-7}}{23040 } $.
\item $\alpha_9(n) =  \frac{n - 15 n^{-1} + 98 n^{-3} - 310 n^{-5} + 381 n^{-7} - 155 n^{-9}}{10240 } $.
\end{itemize}

The following result follows directly from the definitions of $\phi$ and $\psi_k$.
\begin{lemma} \label{lem:mod4}
If $n\equiv 0\pmod{4}$, then $\phi(n) = 2\phi(\frac{n}{2})$ and $\psi_k(n)=\psi_k(\frac{n}{2})$. 
If $n\equiv 2\pmod{4}$, then $\phi(n) = \phi(\frac{n}{2})$ and $\psi_k(n)=(1-2^k)\psi_k(\frac{n}{2})$.
\end{lemma}
\begin{proof}
If $n\equiv 0\pmod{4}$, then $n$ and $\frac{n}{2}$ are both even and thus $n$ and $\frac{n}{2}$ have the same distinct prime factors and this implies that $\psi_k(n)=\psi_k(\frac{n}{2})$.
If $n\equiv 2\pmod{4}$, then the distinct prime factors of $n$ is $2$ plus the distinct prime factors of $\frac{n}{2}$ and thus $\psi_k(n)=(1-2^k)\psi_k(\frac{n}{2})$.
Since $\phi(n) = n\psi_{-1}(n)$, the results for $\phi$ follow from the results for $\psi_{-1}$.
\end{proof}

The following Lemma follows from similar arguments as in \cite{Baum1982}.
\begin{lemma} \label{lem:tildeS}
If $n=2$ or $n\equiv 0 \pmod{4}$, then $\tilde{S}_k(n) = S_k(\frac{n}{2})$.
If $n$ is odd, 
then $\tilde{S}_k(n) = n^k\sum_{d|n}\mu(d) \alpha_k(\frac{n}{d})$.
If $n\neq 2$ and $n\equiv 2 \pmod{4}$, then $\tilde{S}_k(n) = S_k(\frac{n}{2})-2^k\tilde{S}_k(\frac{n}{2}) = S_k(\frac{n}{2}) - n^k\sum_{d|\frac{n}{2}}\mu(d) \alpha_k(\frac{n}{2d})$.
\end{lemma}
\begin{proof}
Since $\tilde{S}_k(2) = S_k(1) = 1$, $\tilde{S}_k(n) = S_k(\frac{n}{2})$ for $n=2$.
If $n\equiv 0 \pmod{4}$, then $n$ and $\frac{n}{2}$ have the same set of prime divisors and thus $\gcd(n,x) = 1$ if and only if $\gcd(\frac{n}{2},x) = 1$
and thus $\tilde{S}_k(n) = S_k(\frac{n}{2})$.
If $n$ is odd, let $N_d^k = \{x^k|1\leq x< \frac{n}{2}, (x,n) = d\}$.
Then $\alpha_k(n) = \frac{1}{n^k}\sum_{d|n}\sum_{x\in N_d^k}x$.
If $x\in N_d^k$ then $\gcd(x,n) = d$ and $\gcd(\frac{x}{d},\frac{n}{d}) = 1$.
Since $1\leq \frac{x}{d}<\frac{n}{2d}$, $\left(\frac{x}{d}\right)^k$ is a term in the sum in $\tilde{S}_k(\frac{n}{d})$ and
$x^k$ is a term in $d^k\tilde{S}_k(\frac{n}{d})$.
The converse is also true, and thus $\sum_{x\in N_d^k}x = d^k\tilde{S}_k(\frac{n}{d}) = \left(\frac{n}{d^*}\right)^k\tilde{S}_k(d^*)$ where
$d^*=\frac{n}{d}$.
Thus $\alpha_k(n) = \sum_{d|n}\frac{\tilde{S}_k(d)}{d^k} $. Applying the M\"{o}bius inversion formula, we get 
$\frac{\tilde{S}_k(n)}{n^k} = \sum_{d|n} \mu(d) \alpha_k(\frac{n}{d})$.
If $n\equiv 2\pmod{4}$, then let $m = \frac{n}{2}$ and it follows that $\gcd(x,n) = 1$ if and only if $\gcd(x,m)=1$ and $x$ is odd. Thus we have
\begin{eqnarray}
\tilde{S}_k(n) &=& \sum_{1\leq x\leq m, (x,m) = 1, x \text{ odd}}x^k\\
&=& \sum_{1\leq x\leq m, (x,m) = 1}x^k - \sum_{1\leq 2x\leq m, (x,m) = 1}(2x)^k \\
&=& S_k(m)-2^k\tilde{S}_k(m)
\end{eqnarray}
\end{proof}

\begin{theorem}\label{thm:main}
If $n\equiv 0\pmod{4}$, then
$$ \tilde{S}_k(n) = \frac{n^k\phi(n)}{2^{k+1}(k+1)} + \sum_{j=2}^{k}\frac{ n^{k+1-j}\binom{k+1}{j}B_j^+\psi_{j-1}(n)}{2^{k+1-j}(k+1)}$$
If $n$ is odd, 
 $$\tilde{S}_k(n) = \sum_{-1\leq i\leq k, i\text{ odd}}b_{-i}^kn^{k-i}\psi_i(n)$$
If $n\neq 2$ and $n\equiv 2\pmod{4}$, then
$$\tilde{S}_k(n) = \frac{n^k\phi(n)}{2^k(k+1)} + \sum_{j=2}^{k}\frac{n^{k+1-j}\binom{k+1}{j}B_j^+\psi_{j-1}(n)}{2^{k+1-j}(k+1)(1-2^{j-1})}-\sum_{-1\leq i\leq k, i\text{ odd}}\frac{b_{-i}^kn^{k-i}2^{i}\psi_i(n)}{1-2^i}$$
where the coefficients $b_i^k$ are given by Eq. (\ref{eqn:bik}).
\end{theorem}

\begin{proof}
If $n\equiv 0\pmod{4}$, then the result follows from Lemma \ref{lem:tildeS} and Eq. (\ref{eqn:sum1a}).
If $n$ is odd, then Lemma \ref{lem:tildeS} implies that
\begin{eqnarray}\tilde{S}_k(n) &=& n^k\sum_{d|n}\mu(d)\sum_{-k\leq i\leq 1, i\text{ odd}} b_i^k(n/d)^i \\
&=& \sum_{-1\leq i\leq k, i\text{ odd}}\sum_{d|n} \mu(d)b_{-i}^k\frac{d^i}{ n^i} \\
&=& \sum_{-1\leq i\leq k, i\text{ odd}}b_{-i}^kn^{k-i}\psi_i(n)
\end{eqnarray}
If $n\neq 2$ and $n\equiv 2\pmod{4}$, then $\frac{n}{2}$ is odd and  
\begin{eqnarray}\tilde{S}_k(n/2) &=& \sum_{-1\leq i\leq k, i\text{ odd}}b_{-i}^kn^{k-i}2^{i-k}\psi_i(n/2)\\
& =& \sum_{-1\leq i\leq k, i\text{ odd}}b_{-i}^kn^{k-i}2^{i}\psi_i(n)/(2^k(1-2^i))
\end{eqnarray}
and the result follows from Lemma \ref{lem:tildeS}.
\end{proof}

Recall that $S_k(n)$ can be written as Eq. (\ref{eqn:sum1a}).
We have shown that $\tilde{S}_k(n)$ can also be written in a similar form as 
\begin{equation}\label{eqn:sum2a}
\tilde{S}_k(n) = \sum_{-1\leq i\leq k, i\text{ odd}} \tilde{a}_{i}^{k} n^{k-i}\psi_{i}(n)
\end{equation}
where the rational coefficients $\tilde{a}_i^k$ for $i\geq 1$ depend on both $k$ and the residue of $n$ modulo $4$ and $\tilde{a}_{-1}^k$ depends only on $k$ and does not depend on $n$. Note however that there is a subtle difference between the structural forms of Eq. (\ref{eqn:sum1a}) and Eq. (\ref{eqn:sum2a}). In Eq. (\ref{eqn:sum1a}) the index $i$ is strictly less than $k$, whereas the index $i$ in Eq. (\ref{eqn:sum2a}) is less than or equal to $k$. Thus for odd $k$, Eq. (\ref{eqn:sum2a}) has one more term than Eq. (\ref{eqn:sum1a}).
This can be seen as $S_1(n)$ has only the term $\psi_{-1}$, whereas $\tilde{S}_1(n)$ has both a term $\psi_{-1}$ and a term $\psi_1$.
Similarly $S_3(n)$ has only the terms $\psi_{-1}$ and $\psi_{1}$, whereas $\tilde{S}_3(n)$ has terms $\psi_{-1}$, $\psi_{1}$ and $\psi_{3}$ (see next Section).

\begin{lemma}
As $n\rightarrow\infty$, $\frac{S_k(n)}{\tilde{S}_k(n)}\rightarrow 2^{k+1}$.
\end{lemma}
\begin{proof}
This follows from the fact that the leading term of $S_k(n)$ is $\frac{n^k\phi(n)}{k+1}$ and the leading term of $\tilde{S}_k(n)$ is $\frac{n^k\phi(n)}{2^{k+1}(k+1)}$.
\end{proof}

\section{Explicit formulas of $\tilde{S}_k(n)$} \label{sec:tildeSk}
For small values of $k$, the equations in Theorem \ref{thm:main} give the following explicit formulas for $\tilde{S}_k(n)$ with $n\neq 2$:
\begin{equation*}
\tilde{S}_1(n) = \left\{
 \begin{array}{ll}
\frac{n \phi(n)}{8} & \text{ if }n \equiv 0 \pmod{4} \\
\frac{n \phi(n) - \psi_{1}(n)}{8} & \text{ if }n \equiv \pm 1 \pmod{4} \\
\frac{n \phi(n) - 2 \psi_{1}(n)}{8} & \text{ if }n \equiv 2 \pmod{4}
\end{array}
\right.
\end{equation*}

\begin{equation*}
\tilde{S}_2(n) = \left\{
 \begin{array}{ll}
\frac{n \left(n \phi(n) + 2 \psi_{1}(n)\right)}{24} & \text{ if }n \equiv 0 \pmod{4} \\
\frac{n \left(n \phi(n) - \psi_{1}(n)\right)}{24} & \text{ if }n \equiv \pm 1 \pmod{4} \\
\frac{n \left(n \phi(n) - 4 \psi_{1}(n)\right)}{24} & \text{ if }n \equiv 2 \pmod{4}
\end{array}
\right.
\end{equation*}

\begin{equation*}
\tilde{S}_3(n) = \left\{
 \begin{array}{ll}
\frac{n^{2} \left(n \phi(n) + 4 \psi_{1}(n)\right)}{64} & \text{ if }n \equiv 0 \pmod{4} \\
\frac{n^{3} \phi(n) - 2 n^{2} \psi_{1}(n) + \psi_{3}(n)}{64} & \text{ if }n \equiv \pm 1 \pmod{4} \\
\frac{7 n^{3} \phi(n) - 56 n^{2} \psi_{1}(n) + 8 \psi_{3}(n)}{448} & \text{ if }n \equiv 2 \pmod{4}
\end{array}
\right.
\end{equation*}

\begin{equation*}
\tilde{S}_4(n) = \left\{
 \begin{array}{ll}
\frac{n \left(3 n^{3} \phi(n) + 20 n^{2} \psi_{1}(n) - 8 \psi_{3}(n)\right)}{480} & \text{ if }n \equiv 0 \pmod{4} \\
\frac{n \left(3 n^{3} \phi(n) - 10 n^{2} \psi_{1}(n) + 7 \psi_{3}(n)\right)}{480} & \text{ if }n \equiv \pm 1 \pmod{4} \\
\frac{n \left(21 n^{3} \phi(n) - 280 n^{2} \psi_{1}(n) + 64 \psi_{3}(n)\right)}{3360} & \text{ if }n \equiv 2 \pmod{4}
\end{array}
\right.
\end{equation*}

\begin{equation*}
\tilde{S}_5(n) = \left\{
 \begin{array}{ll}
\frac{n^{2} \left(n^{3} \phi(n) + 10 n^{2} \psi_{1}(n) - 8 \psi_{3}(n)\right)}{384} & \text{ if }n \equiv 0 \pmod{4} \\
\frac{n^{5} \phi(n) - 5 n^{4} \psi_{1}(n) + 7 n^{2} \psi_{3}(n) - 3 \psi_{5}(n)}{384} & \text{ if }n \equiv \pm 1 \pmod{4} \\
\frac{217 n^{5} \phi(n) - 4340 n^{4} \psi_{1}(n) + 1984 n^{2} \psi_{3}(n) - 672 \psi_{5}(n)}{83328} & \text{ if }n \equiv 2 \pmod{4}
\end{array}
\right.
\end{equation*}

\begin{equation*}
\tilde{S}_6(n) = \left\{
 \begin{array}{ll}
\frac{n \left(3 n^{5} \phi(n) + 42 n^{4} \psi_{1}(n) - 56 n^{2} \psi_{3}(n) + 32 \psi_{5}(n)\right)}{2688} & \text{ if }n \equiv 0 \pmod{4} \\
\frac{n \left(3 n^{5} \phi(n) - 21 n^{4} \psi_{1}(n) + 49 n^{2} \psi_{3}(n) - 31 \psi_{5}(n)\right)}{2688} & \text{ if }n \equiv \pm 1 \pmod{4} \\
\frac{n \left(93 n^{5} \phi(n) - 2604 n^{4} \psi_{1}(n) + 1984 n^{2} \psi_{3}(n) - 1024 \psi_{5}(n)\right)}{83328} & \text{ if }n \equiv 2 \pmod{4}
\end{array}
\right.
\end{equation*}

\section{Polynomial sums of totatives}
These formulas allow us to compute $\sum_{1\leq i<n, (i,n)=1} P(i)$ or $\sum_{i\in R_n} P(i)$ as a function of $n$ where $P$ is a polynomial. For instance, consider the sequence where $a(n)$ is defined as the sum of the areas of all rectangles with coprime integer length $l$ and width $w$ and perimeter $2n$. 
For $n=5$, the only such rectangles of semiperimeter $5$ are $1$ by $4$ and $2$ by $3$ and thus the total area is $a(5) = 1\times 4 + 2\times 3 = 10$.
Two rectangles are assumed the same after rotation, so we can assume that $l\leq w$ and thus $l\leq \frac{n}{2}$.
The area of such a rectangle is $lw = l(n-l) = ln-l^2$ and therefore this sequence can be written as $a(n) = \sum_{i\in R_n} ni-i^2 = n\tilde{S}_1(n)-\tilde{S}_2(n)$.
Using the formulas above, we obtain the following explicit formulas for this sequence (OEIS sequence \oeis{A334628}) when $n\neq 2$:
$$
a(n) = \frac{nS_1(n)-S_2(n)}{2}= \frac{n^2\phi(n)-n\psi_1(n)}{12}
$$
Note that even though $\tilde{S}_i$ depends on the residue of $n \pmod{4}$, this formula doesn't since for $n\neq 2$, we have $2a(n) = nS_1(n)-S_2(n)$ which does not depend on the residue of $n \pmod{4}$.

\section{Conclusions}
We study a functional equation and its relationship to Bernoulli numbers and Faulhaber's formula. 
We then show that the sum of powers of elements in the subset of totatives $R_n$ has a similar structural form as the sum of powers of all totatives and can be expressed as a linear combinations of $\psi_i$ for odd $i$ from $-1$ to $k$. We also show that for odd $k$, $\tilde{S}_k$ has one more term than $S_k$. For $k$ much smaller than $n$, these explicit formulas provide an efficient method to computed $\tilde{S}_k$. The corresponding sequences for $1\leq k\leq 4$ are given by OEIS sequences \oeis{A066840}, \oeis{A295574}, \oeis{A295575}, and \oeis{A295576}, respectively.

\end{document}